\documentclass{amsart}

\usepackage{amsmath,amsfonts,amssymb,stmaryrd}
\usepackage{graphicx,psfrag}
\usepackage{graphicx,psfrag,subfigure}
\newtheorem{thm}{Theorem}
\newtheorem{rk}{Remark}
\newtheorem{prop}{Proposition}
\newtheorem{clly}{Corollary}
\newtheorem{lemma}{Lemma}
\newtheorem{defi}{Definition}

\newcommand{\R}{{\mathbb{R}}}
\newcommand{\C}{{\mathbb{C}}}
\newcommand{\Z}{{\mathbb{Z}}}

\newcommand{\D}{{\mathbb{D}}}
\newcommand{\en}{{\mbox{int}}}
\newcommand{\ex}{{\mbox{ext}}}
\newcommand{\fil}{{\mbox{Fill}}}
\newcommand{\fix}{{\mbox{Fix}}}
\begin{document}
\title{Existence of periodic points for self maps of the annulus.}

\author{J.Iglesias, A.Portela, A.Rovella and J.Xavier}

\address{J. Iglesias, Universidad de La Rep\'ublica. Facultad de Ingenieria. IMERL. Julio
Herrera y Reissig 565. C.P. 11300. Montevideo, Uruguay}
\email{jorgei@fing.edu.uy }

\address{A. Portela, Universidad de La Rep\'ublica. Facultad de Ingenieria. IMERL. Julio
Herrera y Reissig 565. C.P. 11300. Montevideo, Uruguay }
\email{aldo@fing.edu.uy }

\address{A. Rovella, Universidad de La Rep\'ublica. Facultad de Ciencias. Centro de
Matem\'atica. Igu\'a 4225. C.P. 11400. Montevideo, Uruguay}
\email{leva@cmat.edu.uy}
\
\address{J. Xavier, Universidad de La Rep\'ublica. Facultad de Ingenieria. IMERL. Julio
Herrera y Reissig 565. C.P. 11300. Montevideo, Uruguay }
\email{jxavier@fing.edu.uy }
\begin{abstract}
Consider a continuous surjective self map of the open annulus with degree $d>1$. It is proved that the number of Nielsen classes of periodic points is maximum possible whenever $f$ has a
completely invariant essential continuum. The same result is obtained in negative degree $|d|>1$ and for just forward invariant essential continua, provided that the continuum is locally connected.  We also deal with the problem of wether there is a representative of each Nielsen class in the filled set of the invariant continuum.
Moreover, if the map extends continuously to the boundary of the annulus and both boundary components are either attracting or repelling, the hypothesis on the existence of the invariant continuum is no longer needed for obtaining all the periodic points in the interior of the annulus.
\end{abstract}
\maketitle

\section{Introduction.}

This paper deals with the existence of periodic points for annulus maps. A map is any continuous endomorphism of the open
annulus $A$. The {\em degree} of an annulus map $f$ is the integer $d$ such that $f_*(n)=d.n$, where $f_*$ is the endomorphism of
the fundamental group $\Pi_1(A)\sim \Z$ induced by $f$.
We are particularly interested in the case where the degree of the map has modulus greater than one.

We will consider two different ways of saying that a map has abundance of periodic orbits. The first one is completeness, and is related to Nielsen equivalence.  The second one is the growth rate inequality which will be discussed shortly.   A map of the
annulus is complete if for every $n$ the number of Nielsen classes of fixed points of $f^n$ is equal to
$|d^n-1|$. This is the maximum possible for a map of degree $d$.
Roughly speaking, a continuous map $f:A\to A$ of degree $d$, $|d|>1,$ is said to be complete if it has as many different Nielsen classes of periodic points as the map $z^d$ acting on the circle. Nielsen equivalence and completeness are explained in the next section.

This notion of completeness is related to a great deal of problems on existence of periodic orbits in surface dynamics.  For example, a complete map always satisfies the growth
rate inequality:
\begin{equation*}
\limsup_{n\to \infty}\frac{1}{n}\ln( \#\{\fix(f^n)\}) \geq \ln (d).
\end{equation*}
\noindent

It is an open problem if the rate inequality holds for differentiable maps of degree $d$ acting on the sphere $S^2$ (Problem 3 posed in \cite{shub2}).

Other results related to the existence of periodic points for $C^0$ maps, (not necessarily homeomorphisms) were obtained by Hagopian in \cite{h}, where some conditions are imposed on plane continua in order to obtain the existence of fixed points for any $C^0$ map defined on it.

Existence of recurrence is mostly present in maps with $|d|>1$ because $|d|>1$ implies an 'expansion' on the direction of a simple nontrivial closed curve. This idea will be used frequently throughout the article.

The covering map $z\to z^d$ defined in the annulus $\{z\in \C : 0<|z|<1\}$ is a trivial example of a map without periodic points in the open annulus, and so the mere existence of periodic orbits becomes of interest.  As covering maps lift to homeomorphisms of the plane, the problem also relates to surface homeomorphisms dynamics, in particular with what is called Brouwer Theory. In \cite{iprx} completeness of annulus coverings of degree $d$, $|d|>1$, was proved under the assumption that an essential continuum is preserved.

%Moreover, a representative of each Nielsen class of periodic points is found in the filled set of the preserved continuum. This relates to Cartwright-Littlewood theory, where one seeks to locate the fixed points that are obtained by a Brouwer theory argument (see \cite{cl}, \cite{brown}, \cite{bell} and \cite{krys}).

In this paper we study the same problem for general continuous maps (not necesarily coverings). Some definitions are needed before stating the results.

We say that a map $f: A \to A$ is complete on $K\subset A$ if $f$ is complete and every Nielsen class has a representative in $K$. Whenever $K$ is a compact subset of the (open) annulus, denote by $\fil(K)$ the union of $K$ with the bounded (relatively compact) components of its complement.

\begin{thm}\label{t1} Let $f: A \to A$ be a surjective degree $d$ map of the annulus, and $K\subset A$ an essential continuum such that $f^{-1}(K) = K$.  If $d>1$, then $f$ is complete on $\fil(K)$.
\end{thm}

The ingredientes for the proof of this Theorem are Lefschetz Theory (to obtain fixed points enclosed by curves with nonzero index), Caratheodory Theory (constructions of crosscuts with determined properties) and Nielsen Theory (existence of fixed points in every Nielsen class is assured by proving that every lift of $f$ to its universal covering has a fixed point).

The proof obtained fails for $d<-1$ and it is easy to construct counterexamples of the proof. It is an open problem if the assertion of the Theorem is true also for negative values of $d$.

The hypothesis $f^{-1}(K)=K$ is also important in the proof.
Assuming that $K$ is locally connected, this assumption and the surjectivity of $f$ can be dropped, and a satisfactory conclusion is obtained.

\begin{thm}\label{t2}
Let $f: A \to A$ be a degree $d$ map of the annulus, and $K\subset A$ a locally connected essential continuum such that $f(K) \subset K$.
If $|d|>1$, then $f$ is complete on $\fil(K)$.
\end{thm}

These statements imply the existence of periodic points in the set $\fil(K)$ and there are examples where the periodic points are not contained in $K$. Posing some conditions on the 'behaviour at infinity' of the map $f$, completeness is obtained without any other assumption. To state these conditions, note that the annulus can be compactified with two points, called ends of $A$.

\begin{thm}\label{t3}
Let $f: A \to A$ be a degree $d$ map of the annulus, where $|d|>1$. Assume that $f$ extends continuously to the ends of $A$. Each one of the following conditions imply that $f$ is complete.
\begin{enumerate}
\item
Both ends of $A$ are attracting.
\item
Both ends of $A$ are repelling.
\end{enumerate}
\end{thm}

See sections 4,5 and 6 where completeness is obtained for maps of the annulus under different sets of hypothesis.
In Section 2 the notion of Nielsen class is explained, in Section 3 the index and Lefschetz fixed point Theorem are introduced without proofs, and some of the lemmata needed in the proofs
of the theorems are included.
Section 4 is devoted to the proof of Theorem \ref{t1}.
In Section 5 the proofs for locally connected $K$ are given, and the last Section contains two simple applications of these results to maps of the sphere.

\section{Nielsen Theory.}

Two fixed points $p$ and $q$ of a map $f:A\to A$ are said {\em Nielsen equivalent} if there exists a curve $\gamma$ joining $p$ to $q$ such that $f\gamma$ is homotopic to
$\gamma$ relative to its endpoints. This is an equivalent relation whose classes are called Nielsen classes. Two fixed points that are very close are equivalent.
Of course the existence of a continuum of fixed points gives just one class of Nielsen equivalent fixed points. If $|d|>1$, then the map $p_d(z)=z^d$ on the annulus $A=\C\setminus\{0\}$ has exactly $|d-1|$ fixed points, and also $|d-1|$ Nielsen classes.

It is not difficult to see that a degree $d$ map of the annulus cannot have more than $|d-1|$ classes of fixed points.
A map $f$ is {\em complete} if for every $n$ the number of Nielsen classes of fixed points of $f^n$ is equal to that of the map $p_d$.
So, asking for a map $f$ to be complete is to ask that is has the maximum possible (non artificial) periodic orbits.  The following proposition is a standard result in Nielsen theory (see Jiang's book \cite{jiang} for further information on the subject) and will be used to prove all the results throughout this paper.  Alternatively, you can see Corollary 1 in \cite{iprx}.

\begin{prop}\label{comp} A degree $d$ map of the annulus has $|d-1|$ Nielsen classes of fixed points if and only if any lift of $f$ to the universal covering of the annulus has a fixed point.\end{prop}

Note that in order to prove completeness of a map $f$ satisfying the hypothesis of Theorem \ref{t1}, Theorem \ref{t2} or Theorem \ref{t3}, it suffices to show that any lift of $f$ has a fixed point.
Indeed, note that in any case, every iterate $f^k$ of $f$ satisfies the same hypothesis as $f$. It follows that $f^k$ has $|d^k-1|$ classes of fixed points.

\section{The index.}\label{3}
Let $\Pi:\R^2\to \R^2\setminus\{0\}$, defined by $\Pi(z)=\exp(2\pi i z)$ denote the universal covering of the punctured plane.
The index of a plane closed curve $\gamma$ defined in an interval $I=[a,b]$ such that $\gamma(t)\neq 0$ for every $t$, is defined as the first coordinate of $\gamma'(b)-\gamma'(a)$, where $\gamma'$ is any lift of $\gamma$ under $\Pi$: that is, $\gamma': [a,b]\to \R^2$ is a curve such that
$\Pi(\gamma'(t))=\gamma(t)$ for every $t$. This does not depend on the choice of the lift. Note that the index is an integer.

We abuse notation and make no difference between the curve and its trace in the plane.

\begin{defi}
\label{index}
Let $\gamma$ be a closed curve and $f$ a continuous map defined on $\gamma$ and without fixed points on $\gamma$. Define the Lefschetz index of $f$ in $\gamma$ as the index of the closed curve $f\circ\gamma-\gamma$. This index will be denoted by $I_f(\gamma)$.
\end{defi}

Let $\gamma$ be a simple closed curve in the plane, so that its complement has two connected components, one is bounded and denoted $\en(\gamma)$ and the other unbounded, denoted $\ex(\gamma)$.

The curve $\gamma$ is positively oriented if the index of $t\to \gamma(t)-p_0$ is $1$ whenever $p_0$ is a point that belongs to $\en(\gamma)$.
If $\gamma$ is positively oriented, then $\en(\gamma)$ is located to the left of $\gamma$.

\noindent
{\bf Lefschetz fixed point Theorem.} {\em
Let $f$ be a continuous self-map of the plane.
\begin{enumerate}
\item
If $\gamma$ is a simple closed curve such that $I_f(\gamma)\neq 0$, then $f$ has a fixed point in the  bounded component of $\R^2\setminus\gamma$.
\item
Let $A$ be an annulus in the plane whose boundary components are positively oriented curves $\gamma_0$ and $\gamma_1$, and assume that $I_f(\gamma_0)\neq I_f(\gamma_1)$, then $f$ has a fixed point in $A$.
\end{enumerate}
}

It is our prupose to use the above result to prove existence of periodic points. We will need some techniques for calculating indexes of curves. The next assertion follows by continuity of $I_f(\gamma)$ on $f$ and $\gamma$ whenever the index is defined.

\begin{lemma}
\label{l1}
Assume that $\gamma$ is a closed curve and that $f_t$ is a homotopy. If no $f_t$ has a fixed point in $\gamma$, then $t\to I_{f_t}(\gamma)$ is constant.

Moreover, if $\gamma_t$ is a free homotopy and $f$ has no fixed points on any $\gamma_t$, then $t\to I_f(\gamma_t)$ is constant.
\end{lemma}

The following is an immediate consequence:

\begin{clly}
\label{l3}
Let $f$ be fixed point free on a simple closed curve $\gamma$.
If $P$ is a finite subset of $\gamma$, then there exist $\epsilon>0$ and $\delta>0$ such that the index of $f'$ on $\gamma'$ is equal to the index of $f$ in $\gamma$
whenever the following conditions hold:
\begin{enumerate}
 \item $\gamma'$ is a simple closed curve such that $(\gamma'\setminus\gamma) \cup (\gamma\setminus\gamma')$ is contained in $V_\delta(P)$, the $\delta$ neighborhood of $P$,
 \item $f'$ is a continuous map such that $f'=f$ outside $V_\delta(P)$,
 \item the image under $f'$ of $V_\delta(P)$ is contained in $V_\epsilon(f(P))$.
\end{enumerate}

\end{clly}

The proof of this result follows from the fact that if $\epsilon$ and $\delta$ are sufficiently small, then both $f$ and $f'$ and $\gamma$ and $\gamma'$ are homotopic.  This previous corollary and the following two criteria is all that is needed to calculate indexes in the sequel.

Assume that $f$ is constant, equal to $p$ and that $\gamma$ is a positevely oriented simple closed curve. The index $I_f(\gamma)$ is equal to
$1$ if $p$ belongs to $\en(\gamma)$, equal to $0$ if $p$ belongs to $\ex(\gamma)$, and not defined when $p\in\gamma$.

We proceed to generalize this statement.

\begin{lemma}
\label{l2}Let $f$ be defined on a positively oriented simple closed curve $\gamma:[0,1]\to \R^2$.
Let $s$ be an open arc of $\gamma$ and assume that the image of $s$ is compactly contained in $\en(\gamma)$. Assume also that $f$ has no fixed points on $\gamma$.
Then, there exists a homotopy $\{f_t : 0\leq t\leq 1\}$ beginning at $f=f_0$, and a time $t'\in (0,1)$ such that the following conditions hold:
\begin{enumerate}
\item
$f_t(s)\subset \en(\gamma)$ for every $t<t'$.
\item
$f_t(s)\subset\ex(\gamma)$ for every $t>t'$
\item
$f_{t'}(s)\subset s$.
\item
For every $t$ it holds that $f_t=f$ in $\gamma\setminus s'$ where $s'$ is an open arc compactly containing $s$, and $f_t(s')\cap\gamma\subset s$.
\end{enumerate}

In addition, for a homotopy satisfying all these properties, it holds that $I_{f_0}(\gamma)=I_{f_1}(\gamma)+1$.
\end{lemma}

\begin{proof}
Let $h_t:\R^2\to\R^2$ be a homotopy satisfying that $h_0=id$, $h_1$ carries $f(s)$ to a point $p$ contained in $\en(\gamma)$, and for every $t$ it holds that
$h_t(\en(\gamma))\subset\en(\gamma)$ and $h_t$ is the identity
in a neighborhood $U$ of the closure of $\ex(\gamma)$. Then $h_t\circ f$ is a homotopy begining at $f$, having no fixed points on $\gamma$, and such that the final map has
$h_1(f(s))=p$.\\
So it can be assumed that $f$ is constant in $s$, say $f(s)=p_0\in \en(\gamma)$. Next, let $p(t)$ be a curve defined in $[0,1]$ such that
$p(0)=p_0$, $p(t)\in \en(\gamma)$ if $t<1/2$, and $p(t)\in \ex(\gamma)$ if $t>1/2$ and $p(1/2) \in s$.
Take an arc $s'$ containing $s$ in its interior, so that $s'\setminus s$ has two connected components whose images under $f$ are contained in the complement of $U$.
Let $f_t$ be a homotopy beginning at $f$, such that following three conditions hold: $f_t(s)=p(t)$, $f_t(s'\setminus s)\cap (\gamma\setminus s)=\emptyset$, and $f_t=f$ in
$\gamma\setminus s'$. Then the conditions of the lemma hold with $t'=1/2$. This proves the first assertion of the lemma.\\

It remains to prove that $I_{f_0}(\gamma)=I_{f_1}(\gamma)+1$, the second assertion of the lemma.
It was explained above that $f_t$ can be assumed to be constant in $s$ for every $t$.
So the problem reduces to estimate the jump of $I_{f_t}(\gamma)$ when $t$ passes through $t'$. The construction made implies that no $f_t$ has fixed points in $\gamma\setminus s$ so that every $f_t$ can be assumed to be equal to $f$ in $\gamma\setminus s'$, where $s'$ is a small neighborhood of $s$ whose image under $f_t$ is contained in $\en(\gamma)$ for every $t<t'$.

Then the jump can be calculated as if $s$ is the segment $(-1,1)$ contained in the real axis, oriented from right to left, $s'$ is the arc $(-2,2)$
and the image of $s$ changes from the lower half plane to the upper one passing through the origin, while the intersection of $f_t(s')$ with the real axis is contained in $s$, and $f_1=f_0$ in the rest of $\gamma$. Moreover, it may also be assumed that $f_t (s')$ is piecewise linear. See Figure \ref{ind}. In this case, the index change over $s'$ can be  calculated directly from the definition.

\begin{figure}
\caption{}
\label{ind}
\begin{center}

\psfrag{-2}{$-2$}\psfrag{1}{$1$}\psfrag{-1}{$-1$}
\psfrag{2}{$2$}
\psfrag{f-2}{$f(-2)$} \psfrag{f2}{$f(2)$}
\psfrag{fs}{$f(s)$}\psfrag{s}{$s$}\psfrag{f1s}{$f_1(s)$}
\psfrag{p}{$\gamma =c_1\cup c_2\cup\cdots\cup c_n$}
\psfrag{ss}{$s^{'}$}
\psfrag{rr}{$f(-2)$} \psfrag{r}{$f(2)$}
\includegraphics[scale=0.2]{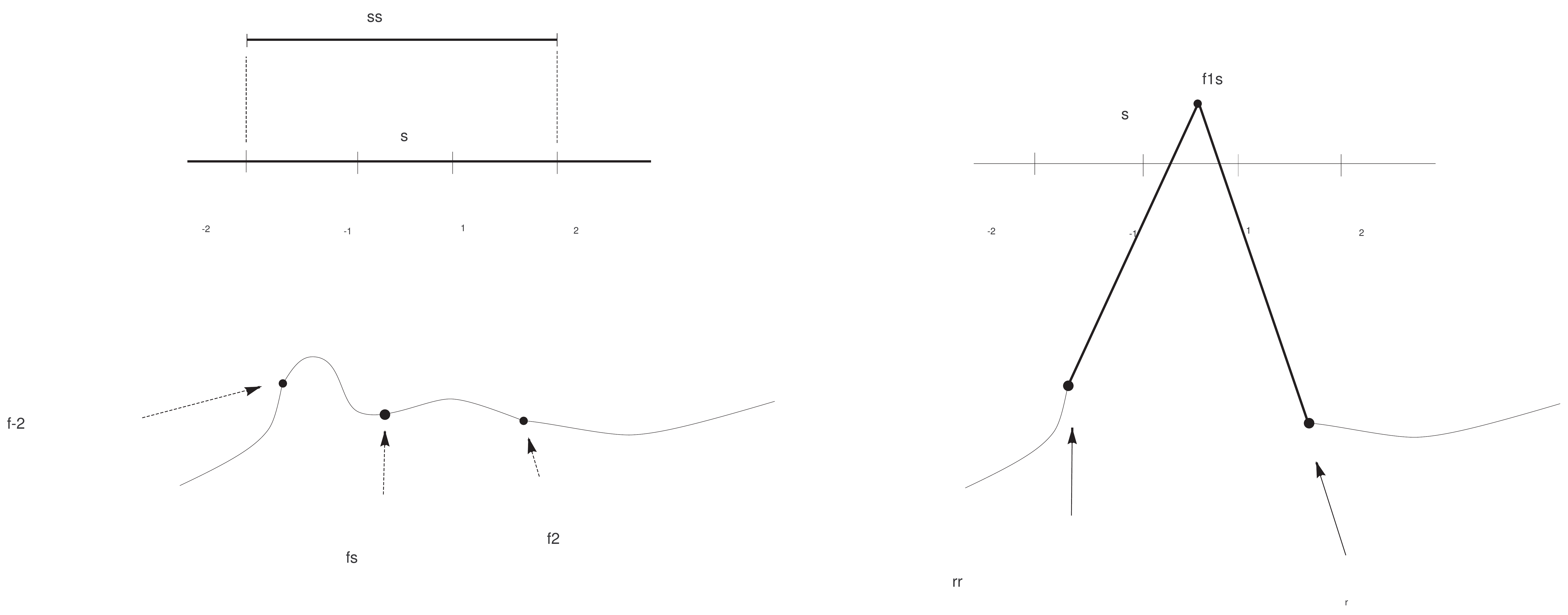}
\end{center}
\end{figure}

\end{proof}

Analogously, we have: 

\begin{clly}\label{meter} Let $f$ be defined on a positively oriented simple closed curve $\gamma:[0,1]\to \R^2$.
Let $s$ be an open arc of $\gamma$ and assume that the image of $s$ is compactly contained in $\en(\gamma)$. Assume also that $f$ has no fixed points on $\gamma$.
Then, there exists a homotopy $\{f_t : 0\leq t\leq 1\}$ beginning at $f=f_0$, and a time $t'\in (0,1)$ such that the following conditions hold:
\begin{enumerate}
\item
$f_t(s)\subset \ex(\gamma)$ for every $t<t'$.
\item
$f_t(s)\subset\en(\gamma)$ for every $t>t'$
\item
$f_{t'}(s)\subset s$.
\item
For every $t$ it holds that $f_t=f$ in $\gamma\setminus s'$ where $s'$ is an open arc compactly containing $s$, and $f_t(s')\cap\gamma\subset s$.
\end{enumerate}

In addition, for a homotopy satisfying all these properties, it holds that $I_{f_0}(\gamma)=I_{f_1}(\gamma)-1$.
\end{clly}

\begin{lemma}
\label{rectangulo}
Let $R\subset \R ^2$ be the square $[-1,1]^2$.  Let $f$ be defined on $\partial R$ such that:
\begin{itemize}
 \item $f(\{y=1\})\subset \{y<1\}$,
 \item $f(\{y=-1\})\subset \{y>-1\}$,
 \item $f(\{x=1\})\subset \{x>1\}$,
 \item $f(\{x=-1\})\subset \{x<-1\}$.

\end{itemize}

Then, $I_f (\gamma) = -1$, where $\gamma$ is $\partial R$ with the positive orientation.

\end{lemma}

\begin{proof}
Let $h_t(x,y)=(x,(1-t)y)$, so that $h_0$ is the identity and $h_1(x,y)=(x,0)$. Let $f_t=h_t\circ f$. Then $f_0=f$ and the hypotheses imply that no $f_t$ has fixed points in $\partial R$ (for example, take any $(x,1)\in \partial R$ and let $f(x,1)=(u,v)$, then $v<1$ implies that
$f_t(x,1)=(u, (1-t)v)$ and $(1-t)v<1$ for every $t\in (0,1)$). Note that the image of $\partial R$ under $f_1$ is contained in the line $y=0$, and that the image of $x=1$ is
contained in $x>1$ and the image of $x=-1$ contained in $x<-1$. It follows from Corollary \ref{meter} that one can obtain a map $f_2$ such that
$f_2(\partial R)$ is contained in the interior of $R$, so that $1=I_{f_2}(\partial R)=I_{f_1}(\partial R)+2= I_{f}(\gamma)+2$.
\end{proof}

Of course, the fact that $R$ is a square is not  essential in the  hypothesis, and we will use the lemma and remark above to calculate index of more general curves:

\begin{clly}
\label{torcido}
Let $\alpha$ and $\beta$ be disjoint simple proper lines in the plane, each one of which separate the plane. Let $\gamma$ and $\delta$ be another pair of disjoint curves separating
the plane. Assume also that each $\gamma$ and $\delta$ intersect $\alpha$ in one point and $\beta$ in one point. Now let $\Gamma$ be the simple closed curve determined by the four
arcs of the curves delimited by the intersection points, with the positive orientation.
Now let $f$ be a map defined on $\Gamma$ such that $f(\Gamma\cap\alpha)$ is contained in the component of the complement of $\alpha$ that contains $\beta$,
$f(\Gamma\cap\beta)$ is contained in the component of the complement of $\beta$ that contains
$\alpha$, that $f(\Gamma\cap\delta)$ is contained in the component of the complement of $\delta$ that does not contain $\gamma$ and that
$f(\Gamma\cap\gamma)$ is contained in the component of the complement of $\gamma$ that does not contain $\delta$. Then the index of $f$ in $\Gamma$ is equal to $-1$.
\end{clly}
\begin{proof}
Note that there exists a homeomorphism $H$ isotopic to the identity such that $H(\alpha)=\{y=1\}$, $H(\beta)=\{y=-1\}$, $H(\gamma)=\{x=-1\}$ and $H(\delta)=\{x=1\}$. Then the map $H\circ f$ satisfies the hypothesis of Lemma \ref{rectangulo} so $-1 = I_{H\circ f}(\partial R)=I_f(\Gamma)$.
\end{proof}

\begin{rk}\label{rk1}
If the last two items are changed to $f(\{x=1\})\subset \{x<-1\}$ and $f(\{x=-1\})\subset \{x>1\}$, then the conclusion is $I_f(\partial R)=1$.
To prove that apply the same homotopy and then observe that there exists $f_2$ homotopic to $f_1$ such that $f_2(\partial R)$ is contained in the interior of $R$, but now the homotopy from $f_1$ to $f_2$ has no fixed points in $\partial R$. The conclusion follows immediately.
\end{rk}

As before, one obtains:

\begin{clly}
\label{torcido2}
Let $\alpha$ and $\beta$ be disjoint simple proper lines in the plane, each one of which separate the plane. Let $\gamma$ and $\delta$ be another pair of disjoint curves separating
the plane. Assume also that each $\gamma$ and $\delta$ intersect $\alpha$ in one point and $\beta$ in one point. Now let $\Gamma$ be the simple closed curve determined by the four
arcs of the curves delimited by the intersection points, with the positive orientation.
Now let $f$ be a map defined on $\Gamma$ such that $f(\Gamma\cap\alpha)$ is contained in the component of the complement of $\alpha$ that contains $\beta$,
$f(\Gamma\cap\beta)$ is contained in the component of the complement of $\beta$ that contains
$\alpha$, that $f(\Gamma\cap\delta)$ is contained in the component of the complement of $\gamma$ that does not contain $\delta$ and that
$f(\Gamma\cap\gamma)$ is contained in the component of the complement of $\delta$ that does not contain $\gamma$. Then the index of $f$ in $\Gamma$ is equal to $1$.
\end{clly}

\section{Proof of Theorem 1.}

\noindent
{\bf Theorem \ref{t1}.} {\em Let $f: A \to A$ be a surjective map with degree $d$, $K\subset A $ an essential continuum such that
$f^{-1}(K)=K$. If $d>1$, then $f$ is complete on $\fil(K)$.}

As was already explained, it is sufficient to prove that every lift of $f$ has a fixed point on the preimage of $\fil(K)$ under the covering projection .

The proof will use indexes of curves, so it is convenient to interpret the annulus as a subset of the plane.
The map $f$ will be considered as a map from $A:=\R^2\setminus \{S\}$ into itself,
where $S$ is the origin (called South; the other end of the annulus is $\infty$, called North).
The map $f$ is assumed to have degree $d$, which is equivalent to impose that the index of the closed curve $t\in[0,1]\to f(\exp(2\pi i t))$ is equal to $d$.

As $K$ is essential, the complement of $K$ in the annulus $A=\R^2\setminus \{S\}$ has two distinguished connected components. The unbounded one is denoted $A_N$, while $A_S$ stands for the connected component of $A\setminus K$ that accumulates on $S$. Note that $A_S\cup \{S\}$ is simply connected, thus Riemann's Theorem implies that there exists a conformal bijection $h$ from $A_S\cup \{S\}$ to the unit disc $\D$ that carries
$S$ to the origin.

\begin{rk}
\label{ends}
Note that as $K$ is connected, then $A_N$ and $A_S$ are the unique components of $A\setminus K$ that are essential in $A$, while the others are simply connected. As the degree of $f$ is nonzero, the image under $f$ of an essential open set is essential, and as $f^ {-1}(K)\subset K$ then $f(A_N)$ is a subset of $A_S$ or a subset of $A_N$. The same can be said about $A_S$. On the other hand, as $f$ was assumed to be surjective, only two possibilities remain: either $f(A_N)\subset A_N$ and $f(A_S)\subset A_S$ or $f(A_N)\subset A_S$ and $f(A_S)\subset A_N$. We will consider the two cases separately.
\end{rk}

We will assume first that each $A_S$ and $A_N$ are $f$-invariant.

\subsection{Constructions in the annulus.}

We begin with some constructions of curves in the annulus, this will not give the desired conclusions, not before lifting the curves to the universal covering and calculating their index under a given lift $F$ of $f$.

We will prove first that $f$ has fixed points in the annulus. Assume by contradiction that $f$ has no fixed point.

It will be convenient to recall some results from Caratheodory's theory. Let $\Omega$ be a simply connected region in the plane with a conformal bijection $h:\Omega\to \D$. A crosscut in $\Omega$ is a simple curve $c$ contained in $\Omega$ except for its extreme points which belong to the boundary of $\Omega$. If $h(c(t))\neq 0$ for every $t$, then the complement of $c$ in $\Omega$ has two connected components, one of which contains $h^{-1}(0)$; the other one is denoted $N(c)$. A point $a$ in the boundary of $\Omega$ is accesible if there exists an arc that is contained in $\Omega$ except for one extreme point in $a$.
It is well known that accesible points are dense in the boundary of $\Omega$. Moreover, given two accesible points there exists a crosscut in $ \Omega$ whose extremes are the
given points. Assume that $\overline\Omega$ is compact, (which in our case is obvious since $A_S\cup \{S\}$ will stand for $\Omega$). Under this assumption, given a positive
number $\delta$, there exists a finite set of accesible points, each one of them can be joined with another by a crosscut of diameter less than $\delta$. These crosscuts will be used to construct an essential simple closed curve.

Let $h$ be a conformal bijection from $A_S\cup\{S\}$ to the unit disc carrying $S$ to the origin $0$.
It is claimed that for every sufficiently small $\delta>0$, there exist a finite sequence of crosscuts $\{c_i:1\leq i\leq k\}$ in $A_S$ such that
\begin{enumerate}
\item[(C1)]
$N(c_i)\cap N(c_j)=\emptyset$ whenever $i\neq j$.
\item[(C2)]
$\gamma=c_1.c_2\ldots c_k$ is an essential positively oriented simple closed curve ,
\item[(C3)]
each $c_i$ has diameter less than $\delta$,
\item[(C4)]
$f(c_i)\cap c_i=\emptyset$.
\end{enumerate}

(See figure 2 a).

\begin{figure}[ht]

\psfrag{kk}{$\tilde{K}$}\psfrag{c2}{$c_2$}\psfrag{c3}{$c_3$}
\psfrag{c1}{$c_1$}
\psfrag{nc1}{$N(c_1)$} \psfrag{k}{$K$}
\psfrag{alpha}{$\alpha$}\psfrag{h}{$h$}\psfrag{g1}{$g_1$}
\psfrag{p}{$\gamma =c_1\cup c_2\cup\cdots\cup c_k$}
\psfrag{v}{$V_{\delta}(P)$}
\par
\begin{center}
\subfigure[]{\includegraphics[scale=0.18]{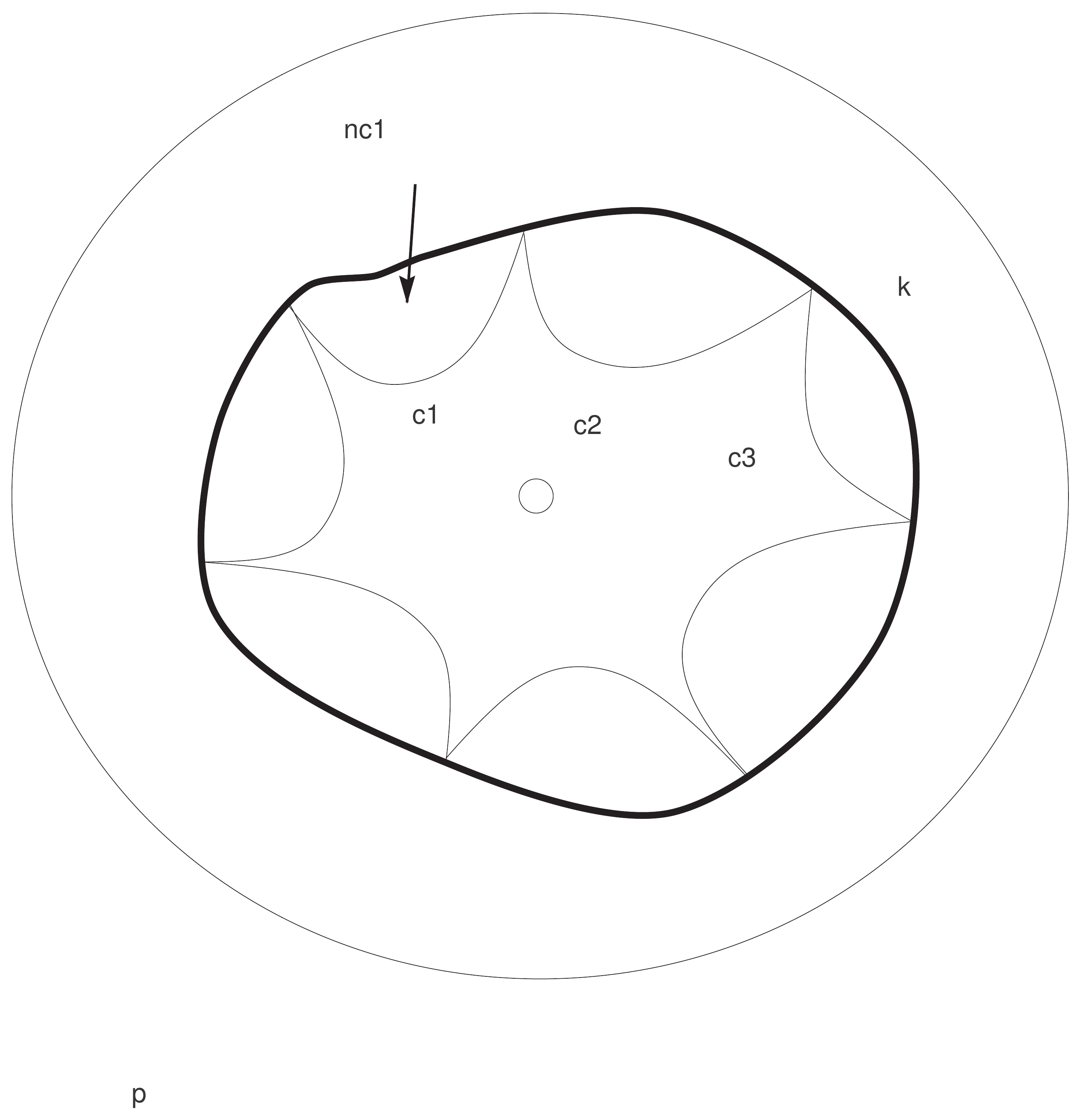}}
\subfigure[]{\includegraphics[scale=0.19]{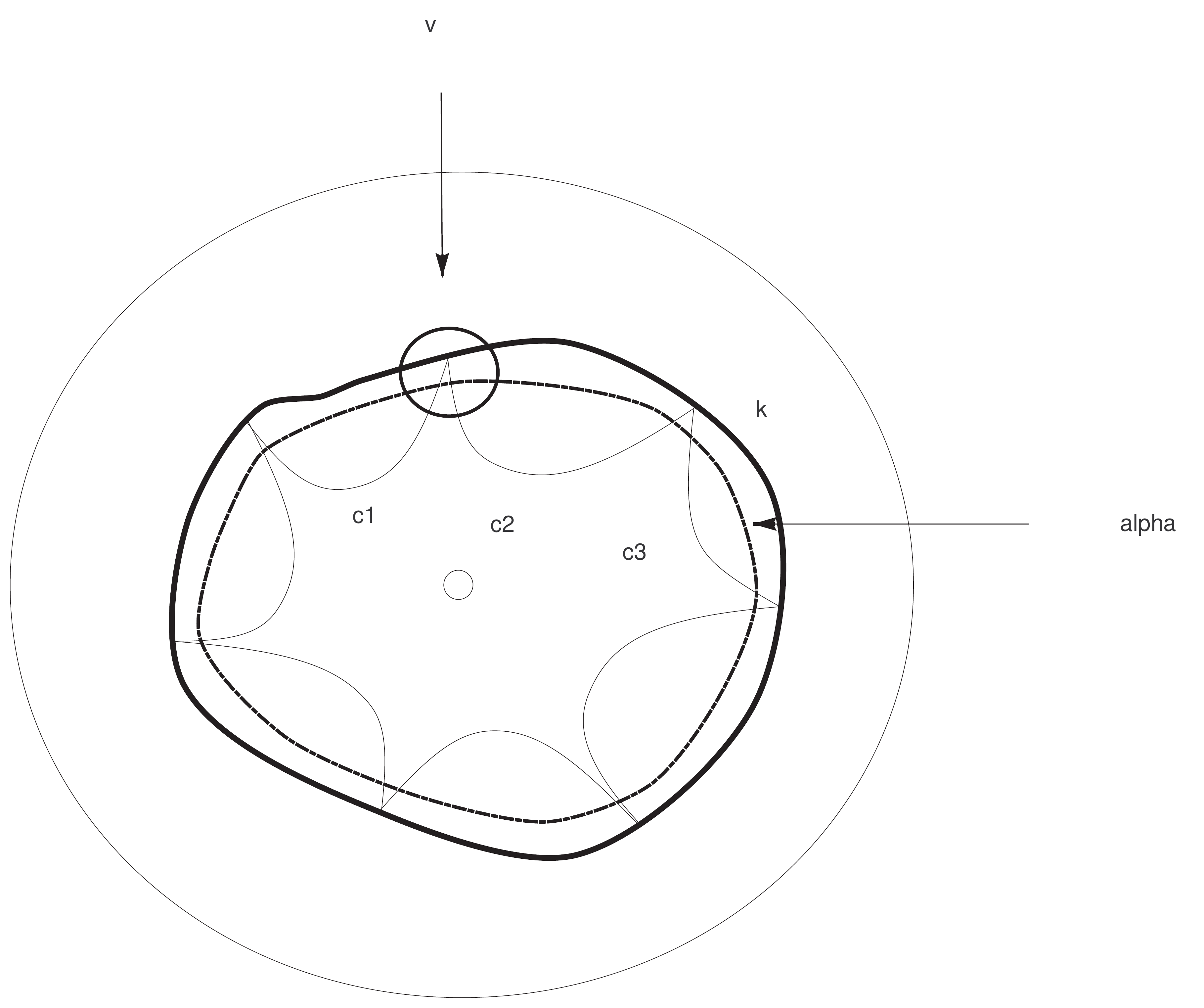}}
\end{center}
\caption{}\label{figura2}
\end{figure}

The proofs of the first, second and third properties follow from the comments above. To prove the last one,
note that the contrary assumption implies that for every sufficiently small positive $\delta_n$ there exists
a point $x_n$ in some crosscut $c_n$ of diameter less than $\delta_n$ such that $f(x_n)\in c_n$, which implies
that the distance from $f(x_n)$ to $x_n$ is less than $\delta_n$. An accumulation point of the sequence $x_n$ must
belong to the boundary of $A_S$, that is contained in $K$, and must be a fixed point of $f$, contrary to our assumption.

Except for the extreme points of the crosscuts, the remaining points of $\gamma$ are contained in the region $A_S$. By assumption in Remark \ref{ends}, $f$ carries $A_S$ to itself. Thus the curve $f(\gamma)$ is contained in $A_S$ excepting for the images of the extreme points of the crosscuts, that are also contained in $K$. Note that the image of a crosscut is not necessarily a crosscut.

Let $P$ be the intersection of $\gamma$ with $K$. $P$ is finite and by Corollary \ref{l3} there exists $\epsilon >0$ and
$\delta >0$ such that one can modify $f$ in the $\delta$ neighborhood of $P$ arbitrarily, whenever the image of the
perturbation is still contained in the $\epsilon$ neighborhood of $f(P)$. Using the fact that $A_S\cup \{S\}$ is conformally
equivalent to the unit disc, it is possible to construct a curve in $A_S$ that is as close to the boundary of $A_S$ as
wished. With this in mind, let $\alpha$ be a simple closed curve contained in $A_S$ that is close to the boundary of $A_S$
in such a way that the intersection of $\alpha$ with $\gamma$ only occurs in $V_\delta(P)$ and that the intersection of
$\alpha$ with $f(\gamma)$ is contained in $V_\epsilon(f(P))$. In order to calculate the index of $f$ in $\gamma$, we apply Corollary \ref{l3}
and we may assume that $\gamma$ and $f(\gamma)$ are both contained in the closure of $A^{\alpha}_S$, the connected component
of the complement of $\alpha$ that contains $S$. Therefore it
can be assumed (for the purposes of index calculation) that the boundary of $A_S$ is a simple closed curve $\alpha$, of course, essential in the annulus $A$.

The next step consists in changing $f$ to a homotopic map $f'$ in order to obtain that $f'(\gamma)\cap\gamma=\emptyset.$ In the process, the index will change, but by means of Lemma \ref{l2} and Lemma \ref{rectangulo} the difference between $I_f(\gamma)$ and $I_{f'}(\gamma)$ will be easily calculated.\\

Using Property (C4), there are two possibilities remaining for each $c_i$: either $f(c_i)\subset N(c_i)$,  or
$f(c_i)\cap N(c_i)=\emptyset$.

Assume that $f(c_i)\subset N(c_i)$.
The image under $f$ of the arc $c_i$ is contained in $\ex(\gamma)$; a direct application of Lemma \ref{l2} implies that the map $f$ is homotopic to a map $f_i'$
satisfying that $f'_i(c_i)\subset \en(\gamma)$ and $I_{f_i'}(\gamma)=I_f(\gamma)+1$ as $ \gamma$ is positively oriented.\\
If there were exactly $J$ of the crosscuts $c_i$ satisfying $f(c_i)\subset N(c_i)$, then change $f$ by a map $f'$ such that $f'(c_i)\cap N(c_i)=\emptyset$ and $I_{f'}(\gamma)=I_f(\gamma)+J$. Moreover $f'$ is equal to $f$ outside a small neighborhood of each $c_i$.\\
Now use $f'$ instead of $f$. Note that, even if $f'(c_j)\cap N(c_j)=\emptyset$ for every $j$, it may very well happen that
$f'(c_j)\cap\gamma\neq\emptyset$.

Let $A^{\gamma}_S$ be the connected component of $A\backslash \gamma$  containing
$S$.  Let $R_t$ be a deformation retract of $A^{\alpha}_S$ to $\gamma_S$ be such that a point in $N(c_i)$ is sent under
application
of $R_1$ to a point in the interior of the arc $c_i$ and such that $R_t(\overline{N(c_i)})\subset \overline{N(c_i)}$ for all $t$. Note that $R_t\circ f'$ cannot have fixed points in $\gamma$ because
for a point $x$ in $c_j$ it is known that $f'(x)$ does not belong to the closure of $N(c_j)$. It follows that the index of
$R_0\circ f'=f'$ and the index of $R_1\circ f'$ with respect to the curve $\gamma$ are equal. Note that the image of
$\gamma$ under $R_1\circ f'$ is contained in the closure of $\en(\gamma)$, so a small $C^0$ perturbation $f''$ of
$R_1\circ f'$ does not change the index and satisfies $f''(\gamma)\subset \en(\gamma)$ and $I_{f''}(\gamma)=1$.

In conclusion, the following assertion was proved.

\begin{clly}
\label{c1}
There exists a map $f'$ homotopic to $f$ such that $f'=f$ in $K$ and all the
components of the complement of $K$ other than $A_S$. Moreover, it holds that
$$
I_f(\gamma)+J=I_{f'}(\gamma)=1,
$$\noindent where $J$ is the number of crosscuts $c_i$ satisfying $f(c_i)\subset N(c_i)$.
\end{clly}

% \begin{proof} It remains to prove the assertion on the index of $f'$, but this is obvious since $\gamma$ is a simple closed curve such that $f'(\gamma)$
% is contained in the bounded component of the complement of $\gamma$.
% \end{proof}

Now, we want to explain briefly what happens if the construction like that of $\gamma$ and the map $f'$ is done in $A_N$, the unbounded component of the complement of $K$ containing $N$. Now $A_N\cup \{N\}$ gives a simply connected region and there exists a Riemann map carrying $A_N\cup N$ to the unit disc and $N$ to $0$. The construction of a curve $\gamma_N$ by concatenation of crosscuts of $A_N$ satisfying the properties (C1) to (C4) is the same as above. Now the component $N(c_i)$ of a crosscut $c_i$ is the bounded component of the complement of $c_i$ in $A_N$. Observe in addition that if the curve $\gamma$ is positively oriented, then a point in
$N(c_i)$ is located in $\en(\gamma)$, therefore by Lemma \ref{l2} $I_f(c_i)=I_{f'}(c_i)+1$ whenever $c_i$ is a crosscut such that $f(c_i)\subset N(c_i)$ and $f'$ is homotopic to $f$ and
constructed exactly as in the previous part. The construction of the curve $\alpha$ is analogous in this case.  The only difference appears when calculating $I_{f'}(\gamma)$:

\begin{clly}
\label{c2}
There exists a map $f'$ homotopic to $f$ such that $f'=f$ in $K$ and all the components of the
complement of $K$ other than $A_N$. Moreover, it holds that
$$
I_f(\gamma)-J'=I_{f'}(\gamma)=d,
$$
where $J'$ is the number of crosscuts $c_i$ such that $f(c_i)$ is contained in $N(c_i)$ and $\gamma = c_1.c_2.\ldots c_n$ with the positive orientation.

\end{clly}
\begin{proof} It remains to see why $I_{f'}(\gamma)=d$, but this is a simple calculation based on the fact that $\gamma$ is a simple closed curve nontrivial in the
annulus, that $f'(\gamma)$ has index $d$ and is contained in the unbounded component of $\gamma$.
\end{proof}

Assuming that $d>0$, the conclusion is the following:

\begin{clly}
\label{c3}
Let $f$ be a map of the annulus of positive degree $d>1$. If $K$ is an essential compact subset of the annulus such that
$f^{-1}(K)=K$ and both unbounded components of the complement of $K$ are invariant under $f$, then $f$ has a fixed point in
$\fil(K)$.
\end{clly}
\begin{proof}
Note first that the curves $\gamma_{N}$ and $\gamma_{S}$ (the notation makes reference to the curve $\gamma$
constructed in $A_N$, resp. $A_S$)
are as close to the set $K$ as wished. Both curves are positively oriented, and their indexes differ, because
$$
I_{f}(\gamma_{N})\geq d>1\geq I_f(\gamma_{S}).
$$

Now Lefschetz fixed point Theorem (see Section \ref{3}) shows that $f$ has a fixed point in the annulus bounded by  $\gamma_{N}$ and $\gamma_{S}$. As these curves can be
constructed arbitrarily close to the boundary of $A_S$ and $A_N$, it comes that $f$ has a fixed point in $\fil(K)$.
\end{proof}

This is not sufficient to conclude that $f$ is complete. It just proves that $f$ has a fixed point, also the same
argument shows that every iterate of $f$ has fixed points, but this contains no additional information. We need to prove that every lift of $f$ has fixed points.

\begin{rk}
Adapting the proof of Corollary 3 a well known Theorem due to Cartwright and Littlewood follows:
{\em If an  orientation preserving homeomorphism $f$ of the plane leaves invariant a continuum $K$ not separating the plane, then $f$ has a fixed point in $K$.}\\
Indeed, take a curve like $\gamma$ in the complement of $K$ and the arguments above show that the index $I_f(\gamma)$ is equal to $d+J$, where $d=1$ (as $f$ is a homeomorphism) and $J\geq 0$.
\end{rk}

\subsection{Lifting the construction.}
We will lift the construction in the annulus and consider an arbitrary lift $F$ of $f$.  We will prove that $F$ has a fixed point (in order to prove completeness by Proposition 1).  Assume by contradiction that $F$ is fixed point free.
It is not straightforward to lift the construction. The difficulty comes at first sight, when the preimage of $K$ under the covering projection is not compact. But there are others, as will become clear while calculating the index of a curve.

The universal covering of the annulus is given by the map $\Pi:\C\to \C\setminus\{0\}$, written in complex coordinates as
$\Pi(z)=\exp(2\pi iz)$.
Note that $\Pi(z+1)=\Pi(z)$ so the group of deck transformations is generated by the translation $z\to z+1$.

Let $K'$ be the preimage of $K$ under the covering projection.
Note that $K'$ is not a compact set but is a closed subset of some horizontal strip $-T\leq y\leq T$; we will assume $T=1$ to simplify notation. Moreover,
$K'$ separates the plane in at least two components, one contains $\{y\leq -1\}$, denoted $A'_S$ and the other contains
$\{y\geq 1\}$, denoted $A'_N$. Denote as usual $z=(x,y)$ for the compex number $z=x+iy$.

Let $\gamma$ and $\alpha$ the curves in $A_S$ constructed above and denote by $\gamma ^S$ and $\alpha ^S$ the preimages of $\gamma$ and
$\alpha$ under the covering projection.
These are embedded lines, each one of which separates the plane in two components.

Choose $q_0^S\in \alpha^S\cap \gamma ^S$ and let $\gamma_0 ^S$ be the lift of $\gamma$ that begins at $q_0^S$. For $m\in \Z$, let
$$
\Gamma_m^S= \bigcup_{k=-m}^m (\gamma_0^S + (k,0)).
$$

\noindent We define analogously $\gamma^N, \alpha ^N, q_0^N$ and $\Gamma _m ^N$.
The closed strip contained between the lines $\alpha^N$ and $\alpha^S$ is denoted by $A^0$.
Let $V_0$ be a simple curve from $q_0^S$ to $q_0^N$ contained in $A^0$, and let $V_m = V_0 + (m,0)$.   Let $V_0'\supset V_0$ an embedded line that separates the plane leaving
$\gamma_0^S \cup \gamma_0^N$ in the connected component of $\R^2\backslash V_0'$ to the right of $V_0'$.  Let $V'_m = V'_0 + (m,0)$. Clearly $V_0$ separates the horizontal strip $A^0$.  Note that
as $F(x+1,y)=F(x,y)+(d,0)$, and $d>1$, then there exists $m>0$ such that the following conditions hold:

\begin{enumerate}
\item
The image of $V_m$ under $F$ is contained in the connected component of $\R^2\setminus V'_m$ to the right of $V'_m$.
\item
The image of $V_{-m}$ under $F$ is contained in the connected component of $\R^2\setminus V'_{-m}$ to the left of $V'_{-m}$.
\item
$F(V_{\pm m})\cap \Gamma_m^S = \emptyset$ and $F(V_{\pm m})\cap \Gamma_m^N = \emptyset$
\end{enumerate}

Finally, let $\beta$ be the simple closed curve
$$
\beta=\Gamma_m^N. V_m.(\Gamma_m^S). (V_{-m})
$$
positively oriented.

Recall we have assumed that $F$  has no fixed points. The lift of $\gamma$ is a concatenation of crosscuts, each one satisfiyng
$F(c'_i)\cap c'_i=\emptyset$, where $c'_i$ lifts $c_i$. There exists $F'$, homotopic to $F$, such that $F'(\Gamma_m^S)\cap\Gamma_m^S=\emptyset$. Let $J^S$ be the number of crosscuts contained in $\Gamma_m^S$
such that $F(c'_i)\subset N(c'_i)$. Define $J^N$ analogously. Moreover, note that $F'(\Gamma_m^S)$ is contained in the region $A'_S$, as well as $F'(\Gamma_m^N)\subset A'_N$. By
application of Lemma \ref{l2} it follows that the $I_F(\beta)=I_{F'}(\beta)+J^S+J^N$.

Note that $F'(\beta)\cap\beta=\emptyset$. The fact that $F'(\Gamma_m^N)$ is contained in $A'_N$ and that $F'(\Gamma_m^S)\subset A'_S$ and that $F'=F$ in $V_m\cup V_{-m}$ and the
choice of $m$ imply that $I_{F'}(\beta)=1$ as $\beta$ is contained in $\en F(\beta)$.

This implies that the index of $I_F(\beta)\neq 0$ and this is absurd since we have assumed
that $F$ was fixed point free. It follows that every lift of $f$ has fixed points.

\begin{figure}
\caption{}
\label{levantado}
\begin{center}
\psfrag{t1}{$T=1$} \psfrag{t2}{$T=-1$} \psfrag{kk}{$K^{'}$}
\psfrag{an}{$A^{'}_N$} \psfrag{as}{$A^{'}_S$}
\psfrag{alphan}{$\alpha^N$} \psfrag{alphas}{$\alpha^S$}
\psfrag{vm}{$V_m$} \psfrag{as}{$A^{'}_S$}
\psfrag{v-m}{$V_{-m}$}
\psfrag{p}{$p$}\psfrag{q}{$q$}
\psfrag{gamman}{$\Gamma_m^{S}$}\psfrag{gammas}{$\Gamma_m^{N}$}
\psfrag{vmm}{$V_m^{'}$}\psfrag{fvm}{$F(V_m)$}
\includegraphics[scale=0.2]{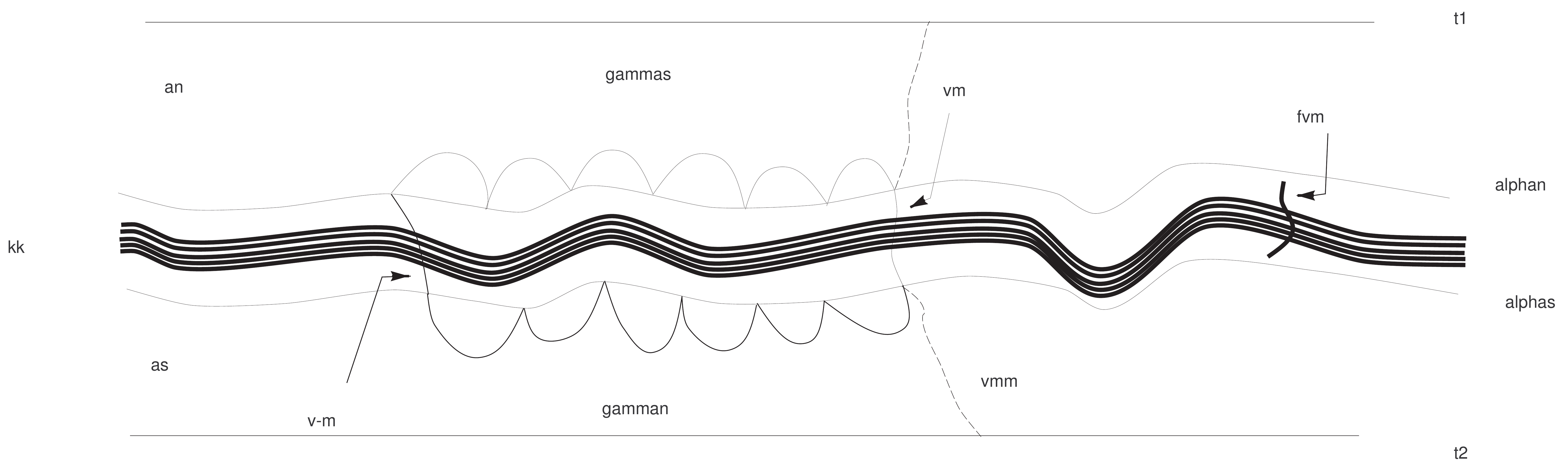}
\end{center}
\end{figure}

Recall that to finish the proof of this theorem, it remains to consider the case where the unbounded components of the complement of $K$ are not invariant, but two-periodic
(Remark \ref{ends}). That is:\\

\noindent
{\em Assume that $f(A_S)=A_N$ and $f(A_N)=A_S$}.\\

\noindent
Let $k$ be an even number. Then $f^k$ satisfies the assumptions made above, (that $A_S$ and $A_N$ are $f^k$-invariant), therefore any lift $G$ of $f^k$ has a fixed point. On the other hand, let $k$ be an odd number and $G$ a lift of $f^k$. Note first that there exists a positive number
$R$ such that $G(\{y=R\})$ is contained in $\{y\leq 0\}$ and $G(\{y=-R\})$ is contained in $\{y\geq 0\}$. By assumption, $d>1$, which implies that there exists a number $m$ such that
$G( m,y)\cap\{x<m\}=\emptyset$ and  $G(- m,y)\cap\{x>-m\}=\emptyset$ whenever $-R\leq y\leq R$.
Let $\Gamma$ be the simple closed curve whose trace is the boundary of the rectangle with vertices $x=\pm m$, $y=\pm R$; give $\Gamma$ the positive orientation. Then $\Gamma$ satisfies the
hypothesis of Lemma  \ref{rectangulo}. Thus the index of $G$ in $\Gamma$ is equal to $-1$.
It follows that $G$ has fixed points. This finishes the proof of Theorem \ref{t1}.

\begin{rk}
\label{cambia}
\begin{enumerate}
\item
It was assumed that $f$ is surjective. It can be seen from the proofs above that it is sufficient to assume that the invariant continuum $K$ is contained in the interior of $f(A)$.
\item
Assume that $f(A_S)=A_N$ and $f(A_N)=A_S$, then the following assertion is also true:\\
{\em Let $f$ be a map of the annulus that interchange the ends (meaning that $f$ can be continuously extended to $A\cup\{S,N\}$ in such a way that $\{S,N\}$ is two-periodic). If the degree of $f$
has absolute value greater than $1$, then for every odd
$k$, the map $f^k$ has at least $|d^k-1|$ fixed points.} For $d>1$ use the arguments above and for $d<-1$, the conclusion follows from Remark \ref{rk1}. In particular, the growth rate
inequality $\limsup_{n\to \infty}\frac{1}{n}\ln( \#\{\fix(f^n)\}) \geq \ln (d)$ holds for $f$.
\end{enumerate}
\end{rk}

\section{Locally connected boundary.}

Let $K\subset A$ be an essential locally connected continuum.
We will use the following properties:

\begin{enumerate}
\item
There exists a simple essential closed curve contained in $K$. See Theorem 43, p.193 in \cite{mo}.
\item
$ K'=\Pi^{-1}(K)$ is connected and locally connected. 
\item Let $U$ be an unbounded connected component of $A\backslash K$. For all $\epsilon >0$ there exists a neighborhood $W$ of $K$ in
$\overline U$ such that the radial retraction $r$ of $W$ onto $K$ satisfies $d(x,r(x))<\epsilon$
for all $x\in W$ (here 'radial' stands for the induced retraction on $W$ by the Riemann map). This is true because $K$ is locally connected and therefore the Riemann map extends continuously to $\overline \D$.
\end{enumerate}

%\begin{figure}[ht]
%
%\psfrag{k}{$K$} \psfrag{tildea}{$A^{'}$} \psfrag{hatk}{$K^{'}$}
%\begin{center}
%\subfigure[]{\includegraphics[scale=0.09]{figura26.eps}} \subfigure[]{%
%\includegraphics[scale=0.22]{figura27.eps}}
%\end{center}
%\caption{}
%\label{ka}
%\end{figure}

\noindent
{\bf Theorem \ref{t2}.} {\em Let $f: A \to A$, $K\subset A $ an essential continuum such that $f(K)\subset K$. If $K$ is
locally connected and $|d|>1$, then $f$ is complete on $\fil (K)$.}

Notations: $\Pi:A'\to A$ is the universal covering of the annulus $A$, the component of the complement of $K$ that contains $N$ is called $A_N$ and
that containing $S$ is denoted $A_S$. Moreover, $A'$ is identified with the product $\R\times (-1,1)$, and define $A'_N=\Pi^{-1}(A_N)$ and $A'_S=\Pi^{-1}(A_S)$.

\begin{proof}
Note that for all $\epsilon>0$ there exists a simple closed curve $\alpha \subset A_N$ such that $d(z,K)<\epsilon$ for every
$z\in \alpha \cup f (\alpha)$.  Analogously, there exists a simple closed curve
$\beta\subset A_S$ satisfying the same properties.  Let $\alpha '$ and $\beta '$ be the preimages of $\alpha$ and $\beta$ by
the covering projection. Note that both $\alpha'$ and $\beta'$ separate $A'$ and as they do not intersect, then there is also a component
of $A'\setminus(\alpha'\cup\beta')$ whose boundary is $\alpha'\cup\beta'$, and this component, denoted $A'_m$, contains $K'$.

Let $F$ be any lift of $f$. Suppose  that
$\fix (F)\cap \fil K' = \emptyset$ (otherwise we are done).  So, there exists $\epsilon >0$ such that $d(x,F(x))>\epsilon$ for all
$x\in \fil K'$. For this $\epsilon$, choose $\alpha$ and $\beta$ as above. By Property (3) of locally connected sets at the beginning of this section, there exists a deformation retract
$r:W\to K'$ such that $d(r, id)<\epsilon/2$, where $W$ is a neighborhood of $K'$.  So, $r\circ F(\alpha')$ is contained
in $K'$.  Analogously,
$r\circ F(\beta')$ is contained
in $K'$.  Besides, the set of
fixed points of
$r\circ F$ in $\fil K'$ are the same as those of $F$. If $d>1$, we may choose real numbers
$x_0< x_1$ such that the intersection of the vertical segment $(x_1,y)$ with $A'_m$ is mapped under $r\circ F$ strictly to its right, and the intersection of the vertical segment
$(x_0,y)$ with $A'_m$ is mapped under $r\circ F$ strictly to its
left.

Let $\gamma$ be the simple closed curve constructed with the segments of $\alpha'$ and $\beta'$ between $x_0$ and $x_1$ and the subarcs of the vertical segments $(x_0,y)$ and $(x_1,y)$ that are contained in $A'_m$.

Note that if $\gamma$ is positively oriented then Corollary \ref{torcido} (of Lemma \ref{rectangulo}) gives $I_{r\circ F}=-1$. This was assuming $d>1$. If $d<-1$, then the
vertical segments at $x_0$ and $x_1$ can be taken so that the vertical at $x_0$ is mapped to the right of $x_1$ and the vertical at $x_1$ is mapped to the left of $x_0$. Now apply Corollary \ref{torcido2} to obtain that the index is equal to $1$.  So, $\fix (F)\cap \fil (K')\neq \emptyset$.

\end{proof}

\noindent
{\bf Example.} We will show that if $d\leq -1$ it may happen that $f$ is fixed point free on $K$. Let $K$ be as in Figure 1 (a). We will construct a map
$F: A' \to  A'$ such that $F(x+1) = F(x)-1$. Take a fundamental domain $D$ of $K'$ as in Figure 1 (b) and we proceed to define $F$ on $D$.
Let $j:[0,1]\to D $ a parametrization such that:
$j(x)\neq j(y)$ for all $\{x,y\}\neq \{a,b\}$, $x\neq y$, and such that $j(a)= j(b)$.
First define $F_0: [0,1]\to [0,1]$ (discontinuous) as in Figure \ref{figura3} (c).  Note that there exists  $F: D\to D$ such that $jF_0(t)= F j(t)$.  Moreover $F$ is continuous on
$D$, as
the disconitnuity of $F_0$ can be avoided on $j([0,1])$ because $j(a)=j(b)$.  Clearly, $F_0$ is fixed point free.  If $F$ has a fixed point $x=j(t)$, then
$$x= F(x)= Fj(t)= jF_0(t).$$  So, $jF_0(t) = j(t)$ implying $t= F_0(t)$ because $t$ and $F_0(t)$ are different from $a$ and $b$.

Note that $F$ factors
to a map $f: A \to A$ of degree $-1$ and that it is fixed point free by construction.
We have then proved that with $d= -1$ one can avoid having a fixed point on $K$. Note that with exactly the same technique one can build a map
$f$, not complete in $K$, of any given
degree $d<-1$.  \\

\begin{figure}[ht]

\psfrag{0}{$0$}\psfrag{a}{$a$}\psfrag{b}{$b$}
\psfrag{1}{$1$}
\psfrag{j}{$j$} \psfrag{j0}{$j(0)$}
\psfrag{ja}{$j(a)=j(b)$}\psfrag{j1}{$j(1)$}\psfrag{k}{$K$}
\psfrag{p}{$\gamma =c_1\cup c_2\cup\cdots\cup c_n$}
\psfrag{ff}{$F_{0}$}
\psfrag{d}{$D$}
\par
\begin{center}
\subfigure[]{\includegraphics[scale=0.13]{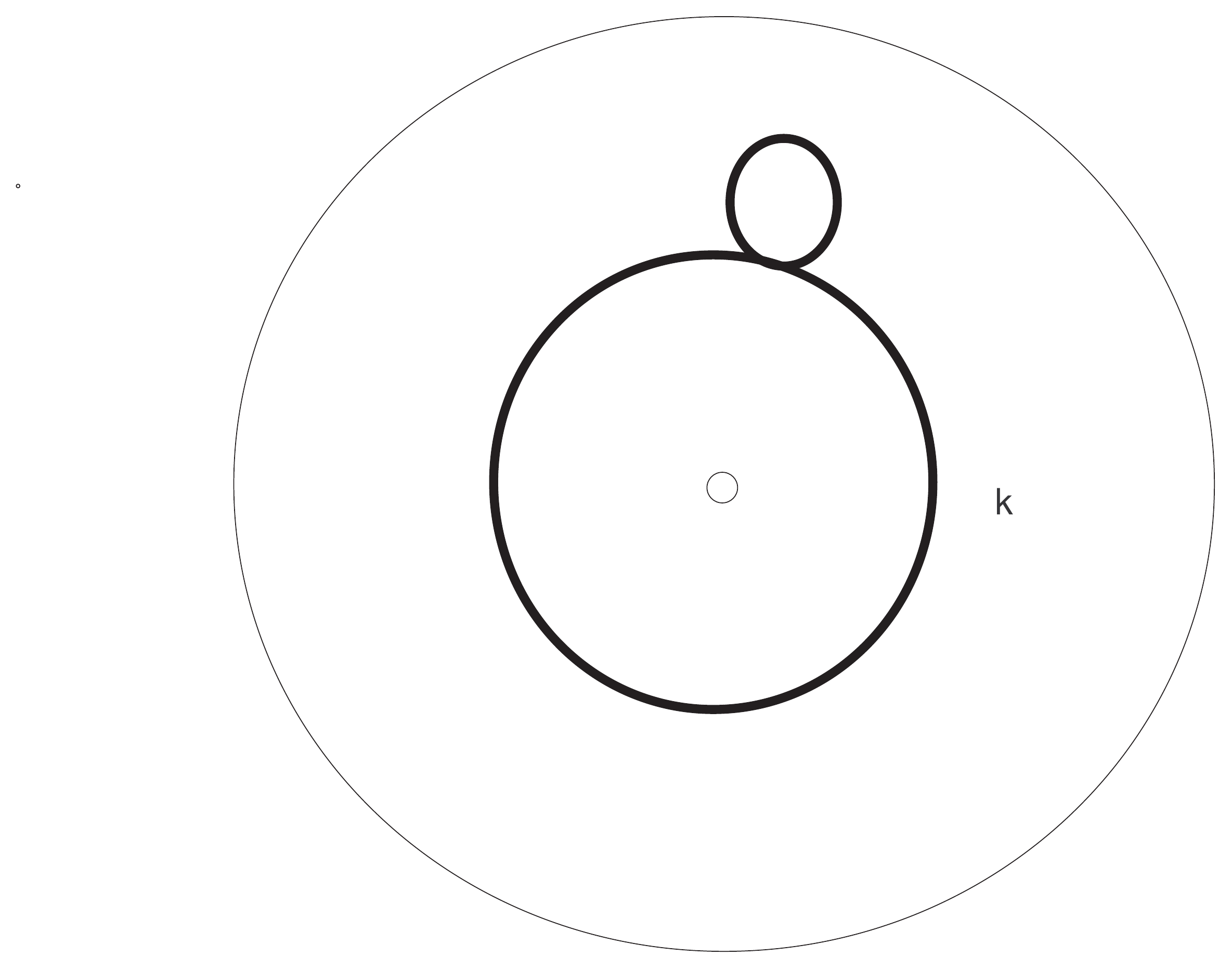}}
\subfigure[]{\includegraphics[scale=0.17]{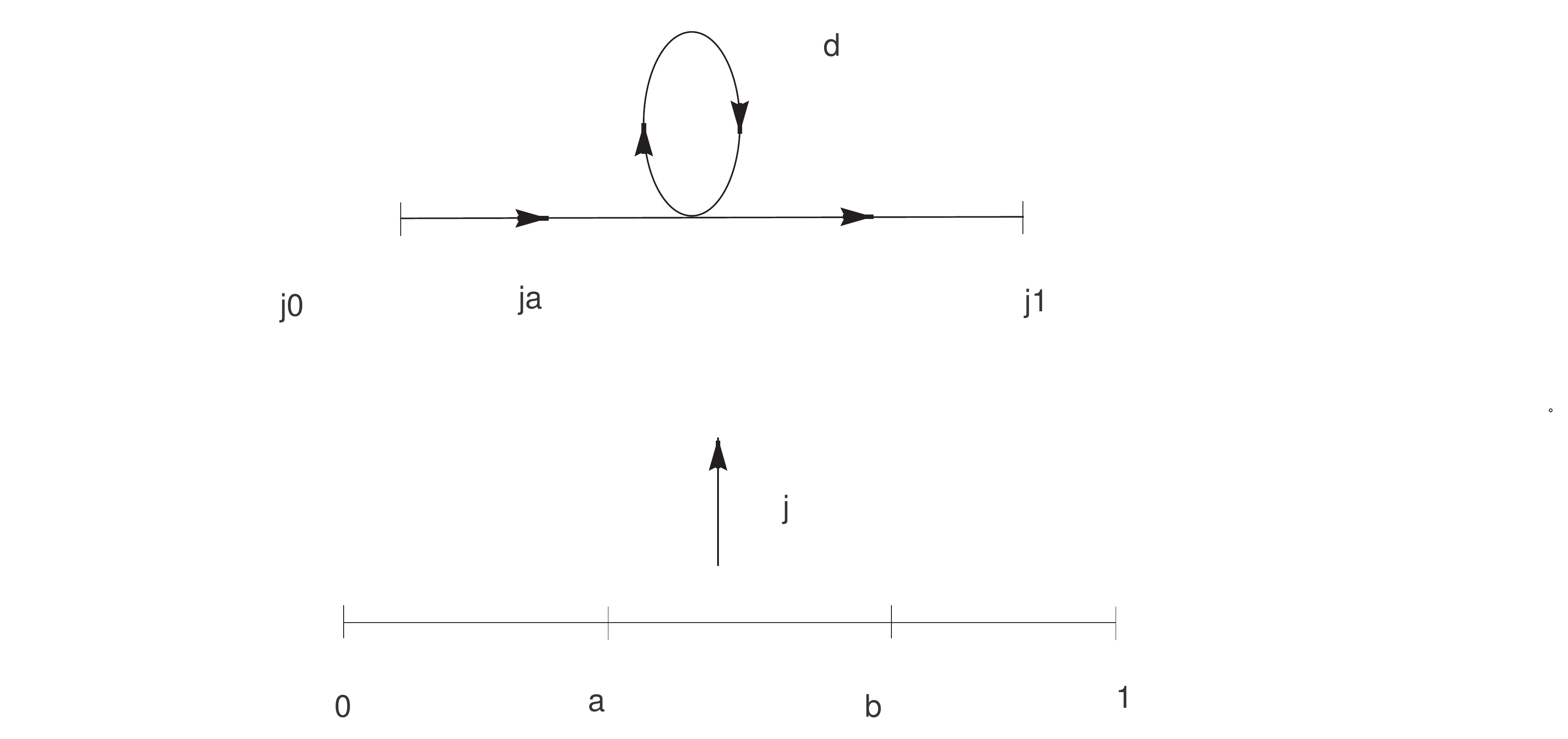}}
\subfigure[]{\includegraphics[scale=0.13]{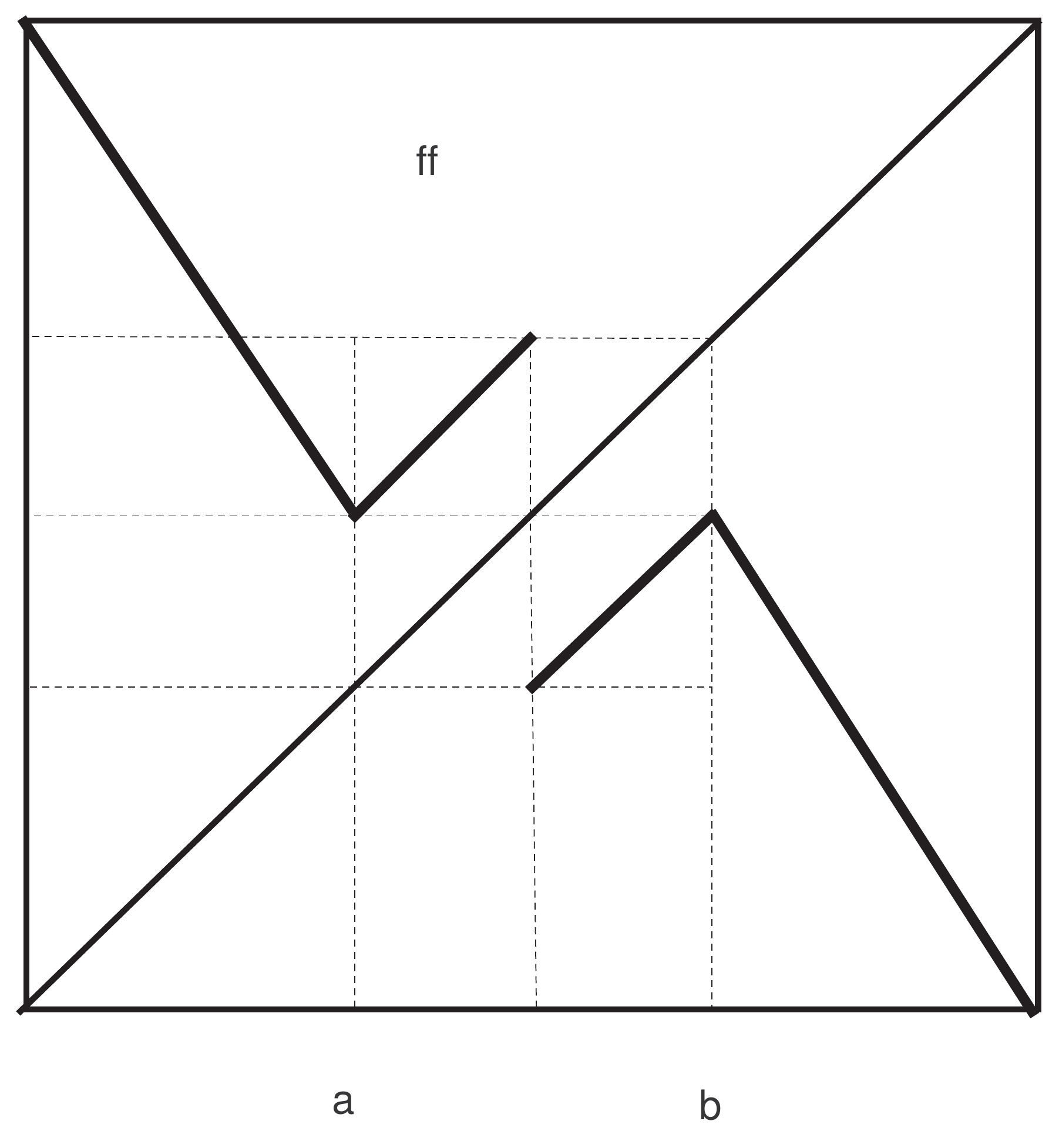}}

\end{center}
\caption{}\label{figura3}
\end{figure}

To assure the existence of periodic points in $K$, and not just $\fil(K)$, we assume in the next Theorem that $K$ is the boundary of one of the unbounded components of its complement and that $d>1$.

\begin{thm}
\label{t4} Let $f: A \to A$ be a map of degree $d$, $K\subset A$ a locally connected essential continuum such that $f(K)\subset K$.
Assume that $K$ is the boundary of one of the unbounded connected components of its complement. If $d>1$, then $f$ is complete on $K$.
\end{thm}

\begin{proof}

By hypothesis $K'$ is the boundary of a simply connected region $R'= \Pi ^{-1}(R) $.\\
The preimage under $\Pi$ of the curve given by property (1) at the beginning of this section,
gives a continuous $1$-periodic function $i:\R\to K'$ with the following properties:
\begin{enumerate}
\item  $i$ is a homeomorphism from $\R$ onto $i(\R)$.
\item $i(\R)$ separates the plane.
\end{enumerate}
\noindent
{\bf Claim.} Let $C$ be a connected component of $ K'\setminus i(\R)$. Then $\overline C$, the closure of $C$, intersects $i(\R)$ in a
one point set.\\
To prove this claim note first that the intersection is not empty because $ K'$ is connected as $K$ is locally connected.
Now assume that there are points $x<y$ such that $i(x)$ and $i(y)$ belong to $\overline C$. As $i(\R)$ is contained in $ K'$
and $K'$ is the boundary of an unbounded region
$ R'$, for all $z\in (x,y)$ there exists a simple embedded line $r\subset  R'$ landing at $i(z)$ and separating
$R'$ in two connected components $A$ and $B$.  As $C$ must have points in $A$ and in $B$, this contradicts the fact that $C$ is connected.\\
From now on the unique point of $i(\R)\cap\overline C$ is denoted by $p(C)$.

Define a map $\psi: K'\to i(\R)$ that is the identity on $i(\R)$ and to each point $x$ in a component $C$ of
$ K'\setminus i(R)$ assigns the point $p(C)$. We claim that as $ K'$ is locally connected, $\psi$ is continuous.
To see this, note that if $z\in i(\R)$, then for all neighborhood $V$ of $z$ there exists a neighborhood $U$ of $z$ such that if $p(C)\in U$ and $p(C)\neq z$ then $C\subset V$.  Otherwise,
there exist a sequence $C_n$ of different components accumulating both on $z$ and in a point $z'$ in $\partial V\backslash i(\R)$. This contradicts
local connectivity. Continuity of $\psi$ now follows.

Define $F'=i^{-1}\psi F i$. This is a continuous function from $\R$ to $\R$.
The assumption $d>1$ implies that there is a point $x_0$ such that $F'(x)<x$ for every $x<x_0$ and a point $x_1$ such that
$F'(x)>x$ for every $x>x_1$.
Let $y_0$ be the supremum of the set of points $y$ such that $F'(y)<y$. As $F'$ is continuous, $F'(y_0)=y_0$ and by
definition of $y_0$, $F'(z)\geq z$ for every $z>y_0$. Then either $i(y_0)$ is fixed for $F$ or $F(i(y_0))$ belongs to a
component $C$ of $K'\setminus i(\R)$ such that $p(C)=i(y_0)$. In the latter case, let $V\subset C$ be a neighborhood
of $Fi(y_0)$ in $K'$ such that $V\cap i(\R)=\emptyset$.
Now let $z>y_0$ be close to $y_0$ in such a way that $Fi(z)\in V$. So $F(i(z))\in C$ and it follows that $F'(z)=y_0$. This is
a contradiction, so the unique conclusion left is $F(y_0)=y_0$.
\end{proof}

\noindent
{\bf Theorem \ref{t3}}. {\em
Let $f: A \to A$ be a degree $d$ map of the annulus, where $|d|>1$. Each one of the following conditions imply that $f$ is complete.}
\begin{enumerate}
\item
Both ends of $A$ are attracting.
\item
Both ends of $A$ are repelling.
\end{enumerate}
\begin{proof}

The hypothesis of the first assertion implies that defining $f(N)=N$ and $f(S)=S$ gives a continuous extension of $f$ to $A^*$, the compactification of the annulus with two points, in such a way that $N$ and $S$ become topological attractors. Then there exists a simple closed essential curve $\alpha$ whose image under $f$ is contained in $\alpha_N$, the connected component of $A^*$ that contains $N$. Similarly, there is a simple closed essential curve $\beta$ whose image under $f$ is contained in $\beta_S$. Now let $F$ be a lift of $f$,
$\alpha'$ and $\beta'$ the preimages of $\alpha$ and $\beta$ under the covering projection, so that there is a region $L\subset A'$ whose boundary is the union of $\alpha'$ and $\beta'$.
Let $V_0$ be a simple arc having one point in $\alpha'$ another point in $\beta'$ and whose interior is contained in $L$,
and let $V_m=V_0+(0,m)$ for $m\in\Z$. Let $\Gamma_m$ be the simple closed curve formed with $V_m$, $V_{-m}$ and parts of $\alpha'$ and
$\beta'$. The fact that $|d|>1$ implies that there exists $m>0$ such that $F(\Gamma_m)$ does not intersect $\Gamma_m$ and that $I_{F}(\Gamma_m)=\pm 1$, depending on $d>0$ or $d<0$.
It follows that $F$ has a fixed point. As every iterate of $f$ satisfies the same hypothesis, then $f$ is complete.

For the case of repelling ends, there exists a curve $\alpha$, close to the end $N$ of the annulus, such that $f(\alpha)$ is contained in the component of the complement of $\alpha$ that does not contain $N$. There will be as well a curve $\beta$, close to $S$, whose image under $f$ is contained in the component  of the
complement of $\beta$ that contains $N$. Next consider a lift $F$ of $f$, take the preimages of $\alpha$ and $\beta$ in the universal covering and cut it with vertical arcs, in such a way to obtain a curve $\gamma$ whose sides are $\alpha'$, $\beta'$ and vertical segments, and satisfies the properties of the curve $\gamma$ of Corollary \ref{torcido} if $d>1$. The Corollary implies that the index of $F$ in $\gamma$ is equal to $-1$. On the other hand, if $d<-1$, then the comments succeding Corollary \ref{torcido} imply that $I_F(\gamma)=1$. This implies that every lift of $f$ has a fixed point. As every $f^k$ satisfies the same hypothesis, it follows that $f$ is complete.

\end{proof}

\section{Applications}

The first result is a consequence of Hagopian's theorem \cite{h} which states that an arcwise connected plane continuum $K$ that does not separate the plane has
the fixed point property (i.e. every continuous map from $K$ to $K$ has a fixed point in $K$). We will use this theorem on the universal covering, so we need that
the lift of K is arcwise connected. This may not happen even if the set K is arcwise connected (see Figure 6).  However, note that in this example the set $K$ is simply connected. If
one assumes that $K$ has a nontrivial closed loop, this guarantees that the lift is arcwise connected.\\

\begin{figure}[ht]

\psfrag{kk}{$\tilde{K}$}\psfrag{c2}{$c_2$}\psfrag{c3}{$c_3$}
\psfrag{nc1}{$N(c_1))$} \psfrag{k}{$K$}
\psfrag{i}{$I$}\psfrag{h}{$h$}\psfrag{g1}{$g_1$}
\par
\begin{center}
\subfigure[]{\includegraphics[scale=0.15]{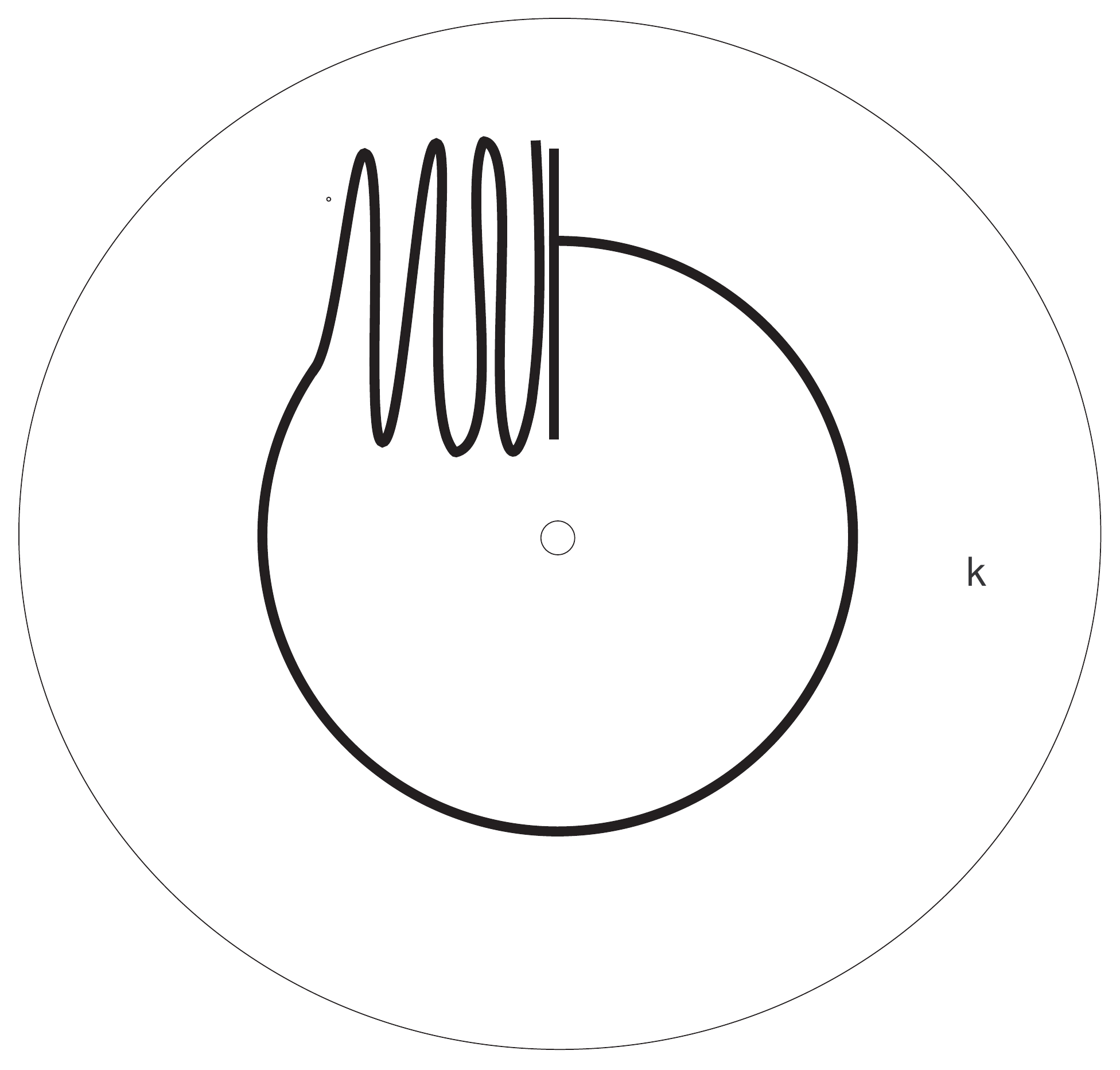}}
\subfigure[]{\includegraphics[scale=0.15]{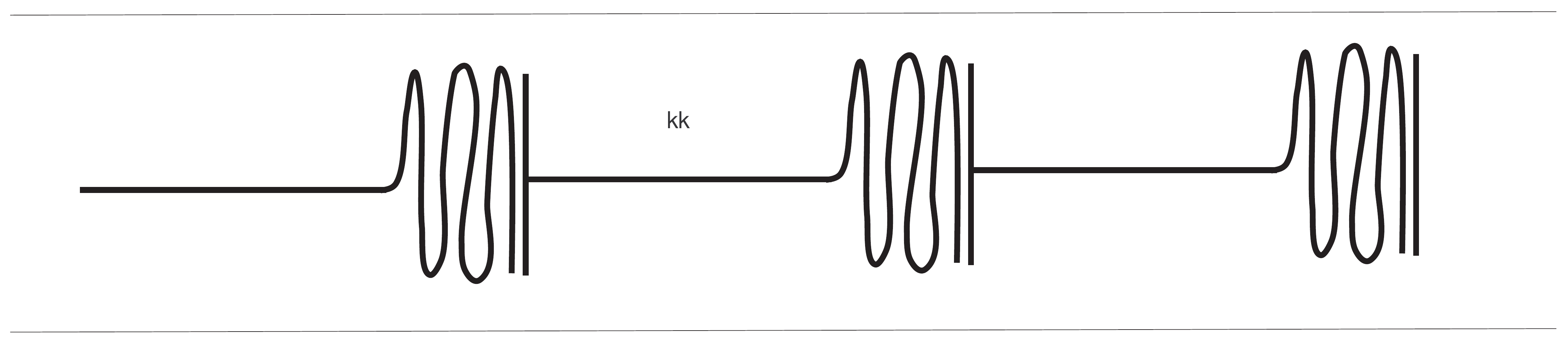}}
\end{center}
\caption{}\label{figura4}
\end{figure}

\begin{prop}
\label{hago}
Let $f: A \to A$, $K\subset A $ a continuum such that $f(K)\subset K$. Assume that $K$ is arcwise-connected, and
$i_* : \pi_1 (K)\to \pi_1 (A)$ is surjective.

If $d<-1$ and $k$ is odd, then every lift $F$ of $f^k$ has a fixed point. It follows that $f$ satisfies the growth rate inequality.

If $d= -1$, then $f$ has a fixed point in $\fil(K)$.
\end{prop}

Compare with Theorem \ref{t1} where it was proved that $f^{-1}(K) = K$ and $d>1$ implies completeness.

\begin{proof}
By hypothesis, $K'$ and $\fil(K')$ are arcwise connected. Take $F$ any lift of $f^k$.
As $k$ is odd, then $f^k$ has negative degree. One can modify $F$ outside $\fil(K')$ so that it can be
extended to $\R\times [0,1]$.  Let $D$ be the compactification of $\R\times [0,1]$ with two points $-\infty, +\infty$,
which is a topological closed disk. As
$d<0$, $F$ interchanges these two points and $\fil(K')\subset D$ is compact.  By Hagopian's Theorem \cite{h}, $F$ must have
a fixed point in $\fil(K')$, and this point does not belong to $\{-\infty, +\infty\}$.
It follows that $F$ has a fixed point in $\fil(K')$. Concluding, if $d<-1$ then every lift of $f^k$ with odd $k$ has a fixed point. This implies that $f^k$ has at least $|d^k-1|$ fixed points, so $f$ satisfies the growth rate inequality. For $d=-1$, just the existence of a
fixed point in $\fil(K)$ can be assured.
\end{proof}

As an application, we obtain a forcing result for $f: S^2\to S^2$: a  completely invariant orbit of period two implies that the growth inequality holds for $f$.

\begin{prop} Let $f: S^2\to S^2$ be a map of degree $d$, $|d|>1$.  Suppose that there exists $p,q \in S^2$ such that $f^{-1}(p) = \{q\}$ and
$f^{-1} (q) = \{p\}$.  Then, $f^k$ has at least $|d^k-1|$ fixed points for every odd $k$.
\end{prop}

\begin{proof}
The restriction of $f$ to $S^2\setminus\{p,q\}$ is a degree $d$ map of the annulus satisfying the hypothesis of the second assertion in
Remark \ref{cambia} after the proof of Theorem \ref{t1}.
\end{proof}

Another result for the sphere:

\begin{prop}  Let $f: S^2\to S^2$ be a map of degree $d>1$.  Assume that there exists two disjoint simply connected nonempty open
sets $U_1$ and $U_2$ such that $f^{-1}(U_i)= U_i$ for $i=1, 2$. Then, $f$ has at least the same number of periodic points as $z^d$ on $S^1$.
\end{prop}
\begin{proof}
Let $K$ be the complement of $U_1\cup U_2$ in $S^2$. Then $K$ is a completely invariant continuum. For $i=1, 2$ choose points $p_i\in U_i$.
Let $f'$ be a map homotopic to $f$ such that $f=f'$ in a neighborhood of $K$ and such that $(f')^{-1}(p_i)=\{p_i\}$ for $i=1, 2$.
Then $f'$ is a degree $d$ map of the annulus $S^2\setminus\{p_1,\ p_2\}$ and has a completely invariant essential continuum $K$. By Theorem \ref{t1} $f'$ is complete on $K=\fil(K)$, so it has at least the same number of periodic points as $z^d$ on $S^1$. The conclusion follows because $f=f'$ on $K$.
\end{proof}

\end{document}